\documentclass[a4paper,11pt]{amsart}
\usepackage{amssymb}
\usepackage{placeins}
\usepackage{a4wide}
\usepackage[shortlabels]{enumitem}
\usepackage{color}

\usepackage{amsmath,amsfonts,amsthm,latexsym}
\newtheorem{thm}{Theorem}
\newtheorem{lem}[thm]{Lemma}
\newtheorem {cor}[thm]{Corollary}
\newtheorem {prop}[thm]{Proposition}

\theoremstyle{definition}

\newtheorem{exa}[thm]{Example}

\theoremstyle{remark}
\newtheorem{rem}[thm]{Remark}

\DeclareMathOperator{\Gal}{Gal}

\DeclareMathOperator{\id}{id}
\DeclareMathOperator{\rad}{rad}
\DeclareMathOperator{\dens}{dens}
\DeclareMathOperator{\N}{N}
\DeclareMathOperator{\Li}{Li}
\DeclareMathOperator{\ind}{ind}
\DeclareMathOperator{\Haar}{Haar}

\newcommand{\p}{\mathfrak{p}}

\newcommand{\Q}{\mathbb{Q}}
\newcommand{\NN}{\mathbb{N}}
\newcommand{\Z}{\mathbb{Z}}
\newcommand{\C}{\mathbb{C}}

\newcommand{\mc}{\mathcal}
\renewcommand{\geq}{\geqslant}
\renewcommand{\leq}{\leqslant}

\newcommand{\av}[1]{\left\lvert#1\right\rvert}
\newcommand{\floor}[1]{\lfloor #1 \rfloor}

\newcommand{\set}[1]{\left\{ #1 \right\}}

\usepackage{xr-hyper}
\usepackage{hyperref}

\begin{document}

\title[Unconditional results for Artin-type problems over number fields]{Unconditional results for Artin-type problems \\ over number fields}

\author{Pietro~Sgobba}
\address{Department of Pure Mathematics, Xi'an Jiaotong--Liverpool University, 111 Ren'ai Road, Suzhou 215123, China}
\email{pietro.sgobba@xjtlu.edu.cn}



\begin{abstract}
Let $K$ be a number field and let $G$ be a finitely generated subgroup of $K^\times$. For all but finitely many primes $\mathfrak p$ of $K$, the reduction $(G \bmod \mathfrak p)$ generates a well-defined subgroup of the multiplicative group of the residue field at $\mathfrak p$, and we may consider its index. We study the  primes of $K$ for which this index lies in a given set of positive integers $S$. In particular, we prove that under certain convergence conditions on series associated to $S$  this problem can be addressed without assuming the Generalized Riemann Hypothesis (GRH), and we provide asymptotic formulas for the corresponding prime-counting functions. Problems of this type are related to Artin’s primitive root conjecture, which has been proven under the assumption of GRH  (Hooley, 1967).
\end{abstract}

\maketitle

\section{Introduction}
Let $K$ be a number field, and let $G$ be a finitely generated and torsion-free subgroup of $K^\times$. Up to discarding finitely many primes $\p$ of $K$, the reduction $(G\bmod\p)$ is the subgroup of $k_\p^\times$, where $k_\p$ is the residue field at $\p$, generated by the reductions $(\alpha\bmod\p)$ for $\alpha\in G$. The multiplicative index of this subgroup, denoted by $\ind_\p(G)$, is well-defined and we may consider the index map
\[ \Psi: \p\longmapsto \Psi(\p):=\ind_\p(G). \]
Given a set $S$ of positive integers, we investigate the natural density of the preimage $\Psi^{-1}(S)$, if it exists. Under the assumption of the Generalized Riemann Hypothesis (GRH), it is known that given a positive integer $t$, the density of the set of primes of $K$ given by $\Psi^{-1}(t)$ exists and has an explicit formula, see \cite[Proposition 1]{zieg}, namely
\[ \dens(\Psi^{-1}(t)) =  \sum_{n\geq1} \frac{\mu(n)}{[K(\zeta_{nt},G^{1/nt}):K]}, \]
where $\zeta_n$ denotes a primitive $n$th root of unity and $\mu$ is the Möbius function.
Also, we consider all extensions of $K$ in a fixed algebraic closure.
Furthermore, again under GRH, the set $\Psi^{-1}(S)$ admits a natural density given by the sum of the densities $\dens(\Psi^{-1}(t))$ for $t\in S$, i.e.
\[ \dens(\Psi^{-1}(S)) =  \sum_{t\in S}\sum_{n\geq1} \frac{\mu(n)}{[K(\zeta_{nt},G^{1/nt}):K]}, \]
see \cite[Corollary 2]{MPS}. Taking $t=1$, the set $\Psi^{-1}(1)$ consists of all primes $\p$ such that $G$ contains a primitive root modulo $\p$, and taking $K=\Q$ and $G$ of rank $1$ reduces to the problem of Artin's primitive root conjecture, which was proven under GRH by Hooley \cite{hooley} in 1967, and its analogue over number fields by Cooke and Weinberger \cite{cooke}; see \cite{MoreeArtin} for an introduction to the research area around this conjecture.  The map $\Psi$ was studied in more generality, under GRH, in two recent articles by Järviniemi and Perucca \cite{JP}, and jointly with the author \cite{JPS}.

Here we investigate the condition $\ind_\p(G)\in S$ without assuming GRH. Instead, we impose an asymptotic bound on a certain series associated with $S$, and then we focus on sets $S$  defined in terms of the $\ell$-adic valuations of their elements, obtaining unconditional partial extensions of the results in \cite{JPS}.

In fact, for some sets $S$ it is well-known that the set of primes $\Psi^{-1}(S)$ can be studied unconditionally. For instance, taking $S=n\NN$ for a given integer $n$, for a prime $\p$ of degree $1$ the condition $n\mid \ind_\p(G)$ is equivalent to $\p$ splitting completely in $K(\zeta_n,G^{1/n})$, and by Chebotarev's density theorem the corresponding natural density is $1/[K(\zeta_n,G^{1/n}):K]$. As a simple application, we may further obtain such examples if we only need to apply this result finitely many times, for instance for the index having a prescribed $\ell$-adic valuation for a fixed $\ell$, or for the index being coprime to a given integer.

In 1995, over the rationals Pappalardi  proved that the primes $p$ for which $\ind_p(2)$ is $k$-free, $k\geq2$, (i.e.\ not divisible by $k$th powers larger than $1$) admit a natural density unconditionally, see \cite[Corollary 5]{pappah}. Notice that differently from the previous examples, this one involves divisibility conditions on the index for infinitely many prime numbers.
In fact, Pappalardi obtained this result as an application of \emph{Hooley's method with weights} \cite{pappah}, which we introduce here in the context of number fields and with an additional condition on the Frobenius conjugacy classes. 

\subsection{Problem setup}\label{sec-setup}
Let $F/K$ be a finite Galois extension of $K$, and let $C$ be a conjugacy-stable subset of $\Gal(F/K)$.
For $\p$ a prime of $K$ which does not ramify in $F$, we denote by $(\p,F/K)$ the conjugacy class of the Frobenius elements at the primes of $F$ above $\p$.
For a given function $f:\NN\to \C$, we consider the sum
\begin{equation}\label{gap}
     \sum_{\substack{\p:\ \N\p\leq x \\ (\p,F/K)\subseteq C}} f(\Psi(\p)),
\end{equation}
where $\N$ is the norm over $K/\Q$ and the sum runs through all primes $\p$ of $K$ except the finitely many primes such that $v_\p(\alpha)\neq0$ for some  $\alpha\in G$ or which ramify in $F$  (we will implicitly discard these primes whenever we consider a condition on $\ind_\p(G)$ or on $(\p,F/K)$, respectively).

Let  $\pi_{n,C}(x)$ be the number of primes $\p$, with $\N\p\leq x$, such that $n\mid \ind_\p(G)$ and $(\p,F/K)\subseteq C$, and write $\pi_{n}(x)$ if the Frobenius condition is trivial. 
Hooley's conditional proof of Artin's conjecture is based on the inclusion-exclusion identity
\begin{equation}\label{eq-hoo}
    \sharp\{\p: \N\p\leq x, \ind_\p(G)=1 \} = \sum_{n\geq1}\mu(n)\pi_n(x).
\end{equation}
By the Möbius inversion formula we may write 
\[ f(n)=\sum_{d\mid n}g(d), \quad \text{ with } \quad  g(n)=\sum_{d\mid n} \mu(n/d)f(d). \]
Then, the sum \eqref{gap} equals
\[ \sum_{\substack{\p:\ \N\p\leq x \\ (\p,F/K)\subseteq C}} \sum_{d\mid\Psi(\p)} g(d)
=  \sum_{d\leq x} g(d) \sum_{\substack{\p:\ \N\p\leq x \\ d\mid\Psi(\p) \\ (\p,F/K)\subseteq C}} 1
 =  \sum_{d\leq x} g(d)\pi_{d,C}(x),  \]
yielding the identity
\begin{equation}\label{eq-mobgen}
    \sum_{\substack{\p:\ \N\p\leq x \\ (\p,F/K)\subseteq C}} f(\Psi(\p)) = \sum_{n\geq1} g(n)\pi_{n,C}(x),
\end{equation}
which is a generalization of \eqref{eq-hoo}; see also \cite[p.376]{pappah} and \cite[Sec.1.3]{felixmu}.

In view of \eqref{eq-mobgen} and since $\pi_{n,C}(x)$ is given asymptotically by Chebotarev's density theorem, we investigate under which conditions we may establish the asymptotic formula
\begin{equation}\label{eq-asymp}
   \sum_{\substack{\p:\ \N\p\leq x \\ (\p,F/K)\subseteq C}} f(\Psi(\p)) \sim \Li(x)\sum_{n\geq1} \frac{g(n)c(n)}{[F(\zeta_n,G^{1/n}):K]}, 
\end{equation}
where $c(n)$ is the number of $K$-automorphisms in $C$ which fix $F\cap K(\zeta_{n},G^{1/n})$ (this coefficient arises when applying Chebotarev's density theorem; we define it formally in Sec.\ \ref{sec-cheb}), and
where we assumed that the series multiplying $\Li(x)$ is convergent. Asymptotic formulas of this form were studied by Felix and Ram Murty \cite{felixmu} over $\Q$ for several arithmetic functions $f$ under GRH, and later by Felix \cite{felixrank} for higher rank. Notice that the constants arising from \eqref{eq-asymp} have been recently studied by Akbary and Fakhari \cite{akbary}.

If  $\chi_S$ denotes the characteristic function of $S$, then taking $f=\chi_S$ the above sum is the counting function for the set of primes $\p$ such that $\ind_\p(G)\in S$ and $(\p,F/K)\subseteq C$. Namely, we have
\begin{equation}\label{eq-count}
    \sum_{\substack{\p:\ \N\p\leq x \\ (\p,F/K)\subseteq C}} \chi_S(\Psi(\p)) =\sharp\set{ \p:\N\p\leq x,\ \ind_\p(G)\in S,\ \begin{pmatrix}\p \\ F/K  \end{pmatrix}\subseteq C }.
\end{equation}

If \eqref{eq-asymp} holds for $f=\chi_S$, from \eqref{eq-count} we would then obtain an asymptotic formula and a density formula for the considered counting function.
Under GRH this was achieved by Felix and Ram Murty \cite{felixmu} over $\Q$, and Felix \cite{felixind} expressed the obtained density and studied its positivity.

Unconditionally, Pappalardi proved that over $\Q$ the asymptotic \eqref{eq-asymp} can be obtained if the series $\sum_{n\geq1}\av{g(n)}/\varphi(n)$ converges (where $\varphi$ is Euler's totient function), and $\sum_{n>z}\av{g(n)}/n=o(1/\log z)$, see \cite[Theorem 1]{pappah}. These conditions are indeed satisfied for $\chi_S$ taking $S$ the set of $k$-free integers.

\subsection{Outline}
Our first main achievement is Theorem \ref{thm_uncond}, which is a generalization and a variation of this result of Pappalardi. It states that if $\sum_{n>z}\av{g(n)}/\varphi(n)\ll z^{-\kappa}$ for some $\kappa>0$, then \eqref{eq-asymp} holds, and it provides an asymptotic formula based on a recent version of Chebotarev's density theorem for cyclotomic--Kummer extensions of number fields \cite[Theorem 10]{sgobba}, also stated here as Theorem \ref{thm-cheb}. 

 In Theorem \ref{thm_kl} we first apply the obtained result to the case of $k$-free indices,  and more generally $k(\ell)$-free indices (i.e.\ not divisible by $\ell^{k(\ell)}$ for all $\ell$), with $k\geq2$ and $k(\ell)\geq2$ for all $\ell$. This illustrates the idea that the series $\sum_{n\geq1}\av{g(n)}/\varphi(n)$ would  converge if $g$ satisfies certain properties implied by the splitting conditions on the primes. For instance, for the index to be squarefree, the primes $\p$ must not split completely in any extension $K(\zeta_{\ell^2},G^{1/\ell^2})$ for all $\ell$. This would then push $g(n)$ to be nonzero if and only if $v_\ell(n)\in\{0,2\}$ for all $\ell$,  where $v_\ell$ is the $\ell$-adic valuation, resulting in the above series being convergent.

We then apply Theorem \ref{thm_uncond} to study the condition $\ind_\p(G)\in H$ for more general sets $H$ whose elements are defined through their $\ell$-adic valuations, for all prime numbers $\ell$. More precisely, we focus on sets $H$ such that $H=\cap_\ell H_\ell$, where $H_\ell$ is the preimage under $v_\ell$ of $v_\ell(H)$. 
We call this type of sets \emph{cut by valuations}. In this case, we assume that $1\in H$ and that for each $\ell$ we have at least $1\in v_\ell(H)$ in Theorem \ref{vlfin} (with some other assumptions), and $1,2\in v_\ell(H)$ in Theorem \ref{vlinf}. In both cases we establish, in particular, the identity
\[ \dens(\Psi^{-1}(H)) = \sum_{n\geq1} \frac{g_H(n)}{[K(\zeta_{n},G^{1/n}):K]}, \]
where $g_H(n)=\sum_{d\mid n, d\in H}\mu(n/d)$ is a multiplicative function.

We then consider slightly more general sets $H$ by allowing finitely many valuation conditions not to be independent, which we call \emph{almost cut by valuations}.
We first study sets $H$ such that $H=v_Q^{-1}(v_Q(H))$, for some fixed $Q\geq1$, where $v_Q(n)$ is the tuple $(v_\ell(n))_{\ell\mid Q}$, see Theorem \ref{thm-HQ}. In other words, elements $n$ of such sets $H$ only have to satisfy properties on $\gcd(n,Q^\infty)$, where $Q^\infty$ is the supernatural number $\prod_{p\mid Q}p^\infty$ ($p$ denotes a prime). The special case where $Q$ is a prime number is treated as an illustrative result in Proposition \ref{prop-l}.
The general case for sets almost cut by valuations is achieved in Theorem \ref{prop-almostcut}. The case $1\notin H$ for sets cut by valuations is thus treated in Remark \ref{rem-relax} as an application of this result.

Finally, in Section \ref{sec-prod} we express the densities for which the convergence condition of Theorem \ref{thm_uncond} is satisfied first as limits and then in Euler product form. In particular, for the latter, in Theorem \ref{thm_prod} and Proposition \ref{prop_prod}, for $H$ almost cut by valuations, we prove the formula $\dens(\Psi^{-1}(H))=D_{K,G,H}\cdot A_{H,r}$, where  $A_{H,r}$ is an Artin-type constant, with $r$ the rank of $G$, e.g.\ if $1\in H$ we have
\[ A_{H,r}=\prod_{\ell} \bigg( 1-\frac{1}{(\ell-1)\ell^r}  + \frac{\ell^{r+1}-1}{(\ell-1)\ell^r}\sum_{\substack{v\geq1 \\ v\in v_\ell(H)}}\frac{1}{\ell^{v(r+1)}} \bigg), \]
 and $D_{K,G,H}$ is interpreted as a correction factor resulting from the entanglement between cyclotomic fields and the Kummer extensions of $K$ related to $G$.

Notice that the above results are already known under the assumption of GRH, in more generality, see  \cite{JP,JPS}. The novelty of this work consists in dealing unconditionally with this general problem under specific convergence conditions associated to the considered set of integers.

 \section{Preliminaries}

\subsection{Notation}\label{sec-not}
Throughout this paper we let $F/K$ be a Galois extension of number fields, $G$ a finitely generated torsion-free subgroup of $K^\times$ of positive rank $r$, and $C$ a conjugacy-stable subset of $\Gal(F/K)$. We also keep all notation from the Introduction.

Given $S$ a set of positive integers, we define
\[ \mc P_{S,C} := \set{ \p: \ind_\p(G)\in S,\ \begin{pmatrix}\p \\ F/K  \end{pmatrix}\subseteq C  }, \]
and we write $\mc P_S$ if the Frobenius condition is trivial, i.e.\ if $C=\Gal(F/K)$.
For a set $\mc P$ of primes of $K$, we write  
\[ \pi_{\mc P}(x):=\sharp\{ \p:\N\p\leq x,\ \p\in\mc P \} \]
for the prime-counting function associated to $\mc P$, and $\dens(\mc P)$ (which we already used in the Introduction) for its natural density, if it exists. According to the notation introduced above, we have  $\pi_{n,C}(x)=\pi_{\mc P_{n\NN,C}}(x)$. We denote by $\pi_{F/K,C}(x)$ the function counting the primes $\p$ of $K$ with norm up to $x$, which are unramified in $F$ and with Frobenius conjugacy class at $\p$ lying in $C$. 
We write $\pi_{F/K}(x):=\pi_{F/K,\{\id\}}(x)$, which counts the primes splitting completely in $F$.
For integers $m,n\geq1$ with $n\mid m$, we set $K_{m,n}:=K(\zeta_m,G^{1/n})$. 

Also, we make use of the following notation: $(m,n)$ is the g.c.d.\ of $m,n$;
$\Li(x)$ is the logarithmic integral function; given $n\geq1$, we write $\rad(n)$ for the radical of $n$, i.e.\ the largest squarefree integer dividing $n$, and $\omega(n)$ for the prime omega function, i.e.\ the number of distinct prime factors of $n$.

\subsection{On Chebotarev's density theorem}\label{sec-cheb}
Here we state an effective and unconditional version of Chebotarev's density theorem for a general cyclotomic--Kummer extension, which was obtained as an application of a recent work of Thorner and Zaman \cite{TZ}.

\begin{thm}[{\cite[Theorem 10]{sgobba}}]\label{thm-cheb} 
For all integers $m,n\geq1$ with $n\mid m$, define
\[
C_{m,n}:=\{\sigma\in \Gal(F_{m,n}/K) : \sigma|_F\in C,\, \sigma|_{K_{m,n}}=\id  \}\,,
\]
which is a conjugacy-stable subset of $\Gal(F_{m,n}/K)$. There exist constants $c_1,c_2>0$, which depend only on $F$ and $G$, such that, for all  $b>0$ and uniformly for
\[
m\leq c_1\left( \frac{\log x}{(\log\log x)^{b}} \right)^{\frac{1}{r+2}}\,,
\]
we have
\[
\pi_{F_{m,n}/K,C_{m,n}} (x)=
\frac{\sharp C_{m,n}}{[F_{m,n}:K]}\Li(x)+
O_{F,G}\left(\frac{x}{\log x}\cdot e^{-c_2(\log\log x)^{b}}\right).
\]
\end{thm}

Taking $F=K$ yields an asymptotic for the function $\pi_{K_{m,n}/K}(x)$  counting the primes splitting completely in $K_{m,n}$. As mentioned in the Introduction  we have $\pi_n(x)=\pi_{K_{n,n}/K}(x)$, and more generally $\pi_{n,C}(x)=\pi_{F_{n,n}/K,C_{n,n}}(x)$. Also, in the latter case we have $\pi_{n,C}(x)\sim \frac{c(n)}{[F_{n,n}:K]}\Li(x)$, where 
\[  c(n):= \sharp C_{n,n} = \sharp \, C\cap\Gal(F/F\cap K_{n,n}) \leq [F:F\cap K_{n,n}]\leq [F:K]. \]
Notice that, since $[F:F\cap K_{n,n}]=[F_{n,n}:K_{n,n}]$, we have
\begin{equation}\label{eq-c}
\frac{c(n)}{[F_{n,n}:K]} \leq \frac{1}{[K_{n,n}:K]},
\end{equation}
with equality holding if  $C\supseteq\Gal(F/F\cap K_{n,n})$.

\begin{prop}[{\cite[Proposition 11]{sgobba}}]\label{prop-BT}
With the notation of Theorem \ref{thm-cheb}, there exists a constant $c_3>0$ (depending only on $F$ and $G$), such that, for all $b>0$ and uniformly for
\[
m\leq c_3\left(\frac{\log x}{(\log\log x)^{1+b}}\right)^{\frac{1}{r+1}},
\]
we have
\[ \pi_{F_{m,n}/K,C_{m,n}}(x)\leq (2+o(1))\frac{\sharp C_{m,n}}{[F_{m,n}:K]}\Li(x). \]
\end{prop}

\subsection{Prescribing $\ell$-adic valuations for a single prime $\ell$}
The following result can be intended as a warm-up problem illustrating how Theorem \ref{thm-cheb} can be used to study Artin-type problems involving divisibility conditions for finitely many distinct prime factors of the index.

\begin{prop}\label{prop-l}
Let $\ell$ be a prime number and let $V\subseteq\Z_{\geq0}$ be a set of integers. The number of primes $\p$ of $K$, with norm up to $x$, such that $v_\ell(\ind_\p(G))\in V$ and $(\p,F/K)\subseteq C$ is given by 
\[ \Li(x)\sum_{v\in V} \left(\frac{c(\ell^v)}{[F_{\ell^v,\ell^v}:K]} - \frac{c(\ell^{v+1})}{[F_{\ell^{v+1},\ell^{v+1}}:K]} \right) + O_{F,K,G}\left( x\left( \frac{(\log\log x)^2}{\log x} \right)^{1+\frac{1}{r+2}} \right). \]
\end{prop}

For trivial Frobenius condition, 
each term in the sum over $v\in V$ corresponds to the density of the primes $\p$ splitting completely in $K_{\ell^{v},\ell^{v}}$ but not in $K_{\ell^{v+1},\ell^{v+1}}$, and, equivalently, of the primes $\p$ such that $v_\ell(\ind_\p(G))=v$. Intuitively, this result does not need the assumption of GRH because the conditions on the index only concern a single prime $\ell$ and the splitting conditions involve number fields of degrees growing fast enough.

\begin{proof}
For $v\geq0$, we set $\mc S_v:=\{\p:v_\ell(\Psi(\p))=v, (\p,F/K)\subseteq C\}$, and $\mc S
:=\bigsqcup_{v\in V}\mc S_v$, i.e.\  the disjoint union of the sets $\mc S_v$ with $v\in V$. 
On the other hand, fixing $w>0$, all primes $\p\in\mc S_v$ with $v\geq w$ must split completely in $K_{\ell^w,\ell^w}$. Hence, we obtain
\[  \pi_{\mc S}(x) =  \sum_{\substack{v\in V \\ v\leq w}} \pi_{\mc S_v}(x)+O\big( \pi_{K_{\ell^w,\ell^w}/K}(x) \big). \]
We set $y:=c_1(\frac{\log x}{(\log\log x)^2})^\frac{1}{r+2}$ (from Theorem \ref{thm-cheb}) and $w:=\floor{\frac{\log y}{\log\ell}}-1$, so that for all $v\leq w$ we have $\ell^{v+1}\leq y$.
By Proposition \ref{prop-BT} we can bound the error term by
\[ \frac{1}{[K_{\ell^w,\ell^w}:K]}\frac{x}{\log x} \ll_{K,G} \frac{1}{\varphi(\ell^w)\ell^{rw}}\frac{x}{\log x} \ll \frac{1}{\ell^{(r+1)w}}\frac{x}{\log x}\leq \frac{x}{y\log x}, \]
where we applied Proposition \ref{kummer}(a), and we used that $\ell/(\ell-1)\leq2$ for all $\ell$, and $w$ is such that $\ell^{(r+1)w}\geq y\geq \ell^w$.

As for the main term, for $v\leq w$ we may apply Theorem \ref{thm-cheb} and obtain
\[ \Li(x)\sum_{\substack{v\in V \\ \ell^{v+1}\leq y}} \left(\frac{c(\ell^v)}{[F_{\ell^v,\ell^v}:K]} - \frac{c(\ell^{v+1})}{[F_{\ell^{v+1},\ell^{v+1}}:K]} \right)
    + O_{F,G}\Bigg( \frac{xy}{\log x \cdot e^{c_2(\log\log x)^2}} \Bigg). \]
The error term is  bounded by $\frac{x}{y\log x}$.
For the tail corresponding to the obtained series, recalling \eqref{eq-c}, we have
\[ \sum_{\substack{v\geq0 \\ \ell^{v}>y}}\frac{c(\ell^v)}{[F_{\ell^v,\ell^v}:K]} \ll_{K,G}
\sum_{\substack{v\geq w}} \frac{1}{\varphi(\ell^v)\ell^{vr}} \ll \frac{1}{\ell^{(r+1)w}}\sum_{v\geq0}\frac{1}{\ell^v} \ll \frac{1}{y}.  \]
Finally, all error terms are estimated by $\frac{x}{y\log x}$, which is bounded by the error term in the statement.
\end{proof}

\subsection{Further preliminaries}

Let us first state some results from Kummer theory.
\begin{prop}\label{kummer}
     There is an integer $B$ (depending only on $F$, $K$ and $G$), such that:
    \begin{enumerate}[(a)]
        \item for all $m,n\geq1$ with $n\mid m$ the ratio
        \[ \frac{\varphi(m)n^r}{[K_{m,n}:K]} \quad \text{ divides } \quad B; \]
\item  for all $n\geq1$ and all prime numbers $\ell\nmid nB$ we have
\[ K_{\ell,\ell}\cap F_{n,n} = K; \]
\item for all $m,n\geq1$  with $(m,nB)=1$ we have
\[ [F_{nm,nm}:K] = [F_{n,n}:K]\varphi(m)m^r. \]
    \end{enumerate}
\end{prop}

\begin{proof}
Assertion (a) follows from \cite[Theorem 1.1]{PS1};
assertion (b) is \cite[Proposition 3.1 (iii)]{JP};
assertion (c) follows from (b) and \cite[Theorem 1.1]{PST}.
\end{proof}

Next we state some analytic results which we will use in our proofs.
\begin{lem}\label{lem-phi}
For any $\alpha\geq1$ we have 
\[ \sum_{n>z}\frac{1}{\varphi(n)n^\alpha}=O\bigg( \frac{1}{z^\alpha} \bigg) \]
\end{lem}

\begin{proof}
    The case $\alpha=1$ is \cite[Lemma 7]{zieg}. The case $\alpha>1$ is a simple consequence:
    \[ \sum_{n>z}\frac{1}{\varphi(n)n^\alpha} \leq \frac{1}{z^{\alpha-1}}\sum_{n>z}\frac{1}{\varphi(n)n} \ll \frac{1}{z^\alpha}. \qedhere \]
\end{proof}

\begin{lem}\label{lem-rankin}
    Let $m\geq1$ be a squarefree integer, let $0<\rho<1$, and $T\geq1$. We have
    \[ \sum_{\substack{a\mid m^\infty \\ a> T}} \frac{1}{a} \ll_\rho \frac{m^{1-\rho}}{T^\rho}. \]
\end{lem}

\begin{proof}
    We have
    \[ \sum_{\substack{a\mid m^\infty \\ a> T}} \frac{1}{a} =  
    m\cdot\sum_{\substack{a\mid m^\infty \\ am> Tm}} \frac{1}{am} \ll_\rho  \frac{m^{1-\rho}}{T^\rho},  \]
    where we applied \cite[Lemma 3.3]{pappa-sf}.
\end{proof}

Occasionally we will also use, without mentioning explicitly, that for all $n$ sufficiently large we have
\[ \frac{n}{\varphi(n)} = \frac{\rad(n)}{\varphi(\rad(n))} \ll \log\log n. \]

\section{The main result}

Let $f,g$ be as in the Introduction, and let us recall that we aim to set convergence conditions associated to $g$ which allow us to establish the asymptotic
\[ \sum_{\substack{\p:\ \N\p\leq x \\ (\p,F/K)\subseteq C}} f(\Psi(\p)) \sim \Li(x)\sum_{n\geq1} \frac{g(n)c(n)}{[F_{n,n}:K]}, \]
where we assume that
\begin{equation}\label{series-conv}
\sum_{n\geq1}\frac{\av{g(n)}}{[K_{n,n}:K]}<\infty. 
\end{equation}
Notice that this assumption is equivalent to each of the following assertions:
\begin{enumerate}[(i)]
    \item $\sum_{n\geq1} \frac{|g(n)|}{\varphi(n)n^r}<\infty$;
    \item $\sum_{n\geq1}\frac{\av{g(n)}c(n)}{[F_{n,n}:K]}<\infty$.
\end{enumerate}
The equivalence with (i) follows from Proposition \ref{kummer}(a), whereas for (ii), the sufficiency follows from \eqref{eq-c}, and the necessity by multiplying each term in the sum \eqref{series-conv} by $\frac{[F:K]c(n)}{[F:F\cap K_{n,n}]}\geq1$.

\begin{lem}\label{lem-bound}
Suppose that $g(n)\geq0$. We have
\[ \sum_{n\geq1} g(n)\pi_{n,C}(x) \geq \bigg(\sum_{n\geq1} \frac{g(n)c(n)}{[F_{n,n}:K]} -o(1)\bigg)\Li(x). \]
\end{lem}

\begin{proof}
Let $y:=c_1\big( \frac{\log x}{(\log\log x)^{b}}\big)^{\frac{1}{r+2}}$ with $b>1$. Applying Theorem \ref{thm-cheb} for $n\leq y$ and Proposition \ref{kummer}(a), we have
\begin{align*}
\sum_{n\geq1} g(n)\pi_{n,C}(x) & \geq \sum_{n\leq y} g(n)\pi_{n,C}(x) \\
& = \Li(x)\bigg(\sum_{n\geq1}\frac{g(n)c(n)}{[F_{n,n}:K]} -O_{F,K,G}\bigg( \sum_{n>y}\frac{g(n)}{\varphi(n)n^ r} + \frac{y^3}{e^{c_2(\log\log x)^b}} \bigg)\bigg). 
\end{align*}
The two terms in the error term are $o(1)$ because of the assumption \eqref{series-conv}. Indeed, for the second term, it follows from \eqref{series-conv} that $g(n)\ll n^2$, and for the limit we also recall that $b>1$.
\end{proof}

\begin{thm}\label{thm_uncond}
Suppose that the function $g$  satisfies the estimate
\begin{equation}\label{eq-conv-new}
    \sum_{n>z} \frac{\av{g(n)}}{\varphi(n)} \ll \frac{1}{z^\kappa}
\end{equation}
for some $0<\kappa\leq2$.
Then
\[ \sum_{n\geq1} g(n)\pi_{n,C}(x) = \Li(x)\sum_{n\geq1} \frac{g(n)c(n)}{[F_{n,n}:K]}+O\left( x\left( \frac{(\log\log x)^{2}}{\log x} \right)^{1+\frac{3\kappa}{2(r+2)}} \right). \]
The constant implied by the error term depends only on $F,K,G,\kappa$ and possibly any other parameters on which the constant implied by \eqref{eq-conv-new} might depend.
\end{thm}

\begin{proof}
Let us set
\[ y:=c_1\left( \frac{\log x}{(\log\log x)^{2}} \right)^{\frac{1}{r+2}}, \quad w:=c_3 \frac{\sqrt{\log x}}{\log\log x}.  \]
As in the proof of Lemma \ref{lem-bound} we have
\begin{align}
\sum_{n\geq1} g(n)\pi_{n,C}(x) =& \Li(x)\sum_{n\geq1}\frac{g(n)c(n)}{[F_{n,n}:K]} + \sum_{n> y} g(n)\pi_{n,C}(x) \label{eq2-pappa}\\
  & + \Li(x)\cdot O_{F,K,G}\bigg( \sum_{n>y}\frac{\av{g(n)}}{\varphi(n)n^ r} + \frac{y^3}{e^{c_2(\log\log x)^2}} \bigg). \label{eq1-pappa}
\end{align}
By the assumption, the first error term resulting from \eqref{eq1-pappa} is $O(\frac{x}{y^{1+\kappa}\log x})$,
and the second error term is negligible compared to this. The second term in \eqref{eq2-pappa} is bounded by
\[ \sum_{n>y} \av{g(n)}\pi_n(x), \]
which we split over the three intervals determined by $y\leq w\leq \sqrt{x}\leq x$. On the interval $[y,w]$ we make use of Proposition \ref{prop-BT} with $b=r=1$ taking into account that $\pi_n(x)\leq \pi_{K(\zeta_n,\alpha^{1/n})/K}(x)$, where $\alpha\in G\setminus\{1\}$. We have
\[ \sum_{y<n\leq w} \av{g(n)}\pi_n(x) \ll_{K,\alpha} \Li(x)\sum_{n>y} \frac{\av{g(n)}}{n\varphi(n)}\ll \frac{x}{y^{1+\kappa}\log x} , \]
in view of our assumption. On the interval $[w,\sqrt{x}]$ we make use of the Brun--Titchmarsh theorem, taking into account that $\pi_n(x)\ll_K\sharp\{p:p\equiv1\bmod n\}$:
\[ \sum_{w<n\leq \sqrt{x}} \av{g(n)}\pi_n(x) 
\ll_K \sum_{w<n\leq \sqrt{x}} \frac{\av{g(n)}x}{\varphi(n)\log(x/n)}
\ll \frac{x}{\log x}\sum_{n>w} \frac{\av{g(n)}}{\varphi(n)}\ll \frac{x}{w^\kappa\log x} . \]
 Finally, on the interval $[\sqrt{x},x]$ we make use of the trivial estimate $\pi_n(x)\ll_K\sharp\{m:n\mid m-1\}$, yielding
\[ \sum_{n> \sqrt{x}} \av{g(n)}\pi_n(x) \ll_K x\sum_{n>\sqrt{x}} \frac{\av{g(n)}}{n}\ll \frac{x}{x^{\kappa/2}}
\ll_{\kappa,e} \frac{x}{y^{e}\log x}, \]
for any $e>0$, where we used that $\varphi(n)\leq n$.

We want to choose a value for $e>0$ in such a way that all terms are bounded by $\frac{x}{y^{e}\log x}$, allowing us to conclude. 
Since $w=y^{\frac{r+2}{2}}$, we may bound $w^{-\kappa}\ll y^{-\frac{3}{2}\kappa}$.
As long as $\kappa\leq 2$ we have $\kappa+1\geq \frac{3}{2}\kappa$, whence we may choose $e=\frac{3}{2}\kappa$. Notice that in all three bounds the implied constant may depend on $\kappa$ or other parameters if this is the case in \eqref{eq-conv-new}.
\end{proof}

\begin{rem}
Let $y,w$ be as in the proof of Theorem \ref{thm_uncond}. It might be useful to point out that in order to reach the same conclusion of the theorem it suffices to have the following three estimates for $y,w$ as above:
\begin{equation}\label{eq_conds}
\sum_{n>y}\frac{\av{g(n)}}{\varphi(n)n}\ll \frac{1}{y^{3\kappa/2}}, \qquad 
\sum_{n>w}\frac{\av{g(n)}}{\varphi(n)}\ll \frac{1}{y^{3\kappa/2}}, \qquad 
\sum_{n>\sqrt{x}}\frac{\av{g(n)}}{n}\ll \frac{1}{y^{3\kappa/2}\log x}.
\end{equation}
These might be less restrictive than requiring \eqref{eq-conv-new} to hold for all $z$ large enough.

Moreover, if in Theorem \ref{thm_uncond} we replace the conditions \eqref{eq-conv-new} or \eqref{eq_conds} with the assumptions that the series $\sum_{n\geq1}|g(n)|/\varphi(n)$ converges and that $\sum_{n\geq z}|g(n)|/n=o(1/\log z)$, similarly as in \cite[Theorem 1(b)]{pappah} we may still deduce the asymptotic
\[ \sum_{n\geq1} g(n)\pi_{n,C}(x) \sim \Li(x)\sum_{n\geq1} \frac{g(n)c(n)}{[F_{n,n}:K]}. \]
Indeed, in this case we may bound the series $\sum_{n>y} \av{g(n)}\pi_n(x)$ over the intervals determined by $y\leq\sqrt{x}\leq x$. Over $[y,\sqrt{x}]$, applying the Brun--Titchmarsh Theorem we have
\[ \sum_{y<n\leq \sqrt{x}} \av{g(n)}\pi_n(x) \ll \frac{x}{\log x}\sum_{n>y}\frac{\av{g(n)}}{\varphi(n)}=o\Big( \frac{x}{\log x} \Big). \]
Over the interval $[\sqrt{x},x]$ we have
\[ \sum_{n> \sqrt{x}} \av{g(n)}\pi_n(x) \ll x\sum_{n>\sqrt{x}} \frac{\av{g(n)}}{n}
=o\Big( \frac{x}{\log x} \Big). \]
\end{rem}

\section{The index in a given set of integers}
In this section, we apply Theorem \ref{thm_uncond} for $f=\chi_S$, where $S$ is a set of positive integers and $\chi_S$ is its characteristic function, so that, recalling the definition of the set $\mc P_{S,C}$ from Sec.\ \ref{sec-not} and the discussion in Sec.\ \ref{sec-setup}, we have
\begin{equation}\label{eq-piS}
\pi_{\mc P_{S,C}}(x)=\sum_{\substack{\p:\ \N\p\leq x \\ (\p,F/K)\subseteq C}}\chi_S(\Psi(\p))=\sum_{n\geq1}g_S(n)\pi_{n,C}(x)
\end{equation}
where
\begin{equation}\label{eq-g}
g_S(n)=\sum_{d\mid n}\mu(n/d)\chi_S(d)=\sum_{\substack{d\mid n \\ d\in S}} \mu(n/d).
\end{equation}
Thus we want to establish conditions on the set $S$ such that we have the asymptotic
\[ \pi_{\mc P_{S,C}}(x) \sim \Li(x) \sum_{n\geq1} \frac{g_S(n)c(n)}{[F_{n,n}:K]}. \]

\begin{rem}\label{rem-dens}
Let $S$ be a set of positive integers such that the function $g_S$ satisfies the convergence condition of Theorem \ref{thm_uncond}. Then, in view of \eqref{eq-piS}, the above asymptotic formula holds, and the set of primes $\p$ of $K$ such that $\ind_\p(G)\in S$ and $(\p,F/K)\subseteq C$ has density 
\[
\dens(\mc P_{S,C}) = 
\sum_{n\geq1} \frac{g_S(n)c(n)}{[F_{n,n}:K]}
=\sum_{n\geq1} \sum_{\substack{d\mid n \\ d\in S}} \frac{\mu(n/d)c(n)}{[F_{n,n}:K]}.
\]
\end{rem}

\subsection{Sets which are cut or almost cut by valuations}

In this section we consider sets of positive integers which are defined through conditions on the $\ell$-adic valuations of the elements, where $\ell$ runs through all prime numbers. 

For an integer $Q>1$  we consider the $Q$-adic valuation $v_Q:\NN\to\prod_{\ell\mid Q}\Z_{\geq0}$ mapping $m$ to the tuple $(v_\ell(m))_{\ell\mid Q}$.
Let $H$ be a set of positive integers, and denote by $H_Q$ the preimage under $v_Q$ of $v_Q(H)$ (we set $H_1:=\NN$). We also set $V_\ell:=v_\ell(H)$ for any prime $\ell$ (notice that $0\in V_\ell$ for all  $\ell$, except finitely many).
We say that $H$ is \emph{cut by valuations} if $H=\bigcap_\ell H_{\ell}$, and \emph{almost cut by valuations} if 
\begin{equation}\label{eq-alm}
    H=H_{Q_0}\cap \bigcap_{\ell\nmid Q_0}H_\ell
\end{equation}
for some $Q_0\geq1$. 
In particular, writing $n=ab$ with $a\mid Q_0^\infty$ and $(b,Q_0)=1$, it follows from the definition that $n\in H$ if and only if $a\in H_{Q_0}$ and $\ell^{v_\ell(b)}\in H_\ell$ for every $\ell\nmid Q_0$.
The above definitions were introduced in a more general context in \cite[Section 3.1]{JPS}.

\begin{lem}\label{lem_mult}
Let $H\subset \NN$. If $H$ is cut by valuations and $1\in H$, then the characteristic function $\chi_H$ and the function $g_H$ are both multiplicative. If $H$ is almost cut by valuations with $Q_0\geq1$ as in \eqref{eq-alm}, then, writing $n=ab$ with $a\mid Q_0^\infty$ and $(b,Q_0)=1$, we have
\begin{equation}\label{eq-chiH}
    \chi_H(n)=\chi_{H_{Q_0}}(n)\prod_{\ell\nmid Q_0}\chi_{H_\ell}(n)
= \chi_{H_{Q_0}}(a) \prod_{\ell\nmid Q_0}\chi_{H_\ell}(\ell^{v_\ell(b)}), 
\end{equation}
\begin{equation}\label{eq-gH}
    g_H(n)= g_{H_{Q_0}}(a) \prod_{\ell\nmid Q_0}g_{H_\ell}(\ell^{v_\ell(b)}).
\end{equation}
We have $\chi_{H_\ell}(1)=g_{H_\ell}(1)=1$ for all $\ell\nmid Q_0$ and such that $0\in V_\ell$.
Moreover, \begin{itemize}
    \item if for all $\ell\nmid Q_0$ we have $0\in V_\ell$, then for all $a\mid Q_0^\infty$ we have $\chi_{H_{Q_0}}(a)=\chi_{H}(a)$ and $g_{H_{Q_0}}(a)=g_{H}(a)$;
    \item  if $1\in H$, then for all $\ell\nmid Q_0$ and $v\geq0$ we have $\chi_{H_{\ell}}(\ell^v)=\chi_{H}(\ell^v)$ and $g_{H_{\ell}}(\ell^v)=g_{H}(\ell^v)$.
\end{itemize}
\end{lem}

\begin{proof}
Suppose that $H$ is cut by valuations and $1\in H$. Clearly we have $\chi_H(1)=1$.  Let $a,b\in \NN$ be coprime. Then $a,b\in H$ is equivalent to $a,b\in H_{\ell}$ for all $\ell$. Since, for all $\ell$, either $v_\ell(a)$ or $v_\ell(b)$ is $0$, the previous statement is equivalent to $v_\ell(ab)\in V_\ell$. In other words, we have $ab\in H_{\ell}$ for all $\ell$, and hence $ab\in H$.
Since the Dirichlet convolution of two multiplicative functions is multiplicative, the function $g_H=\chi_H*\mu$ is multiplicative. 

Suppose that $H$ is almost cut by valuations with $Q_0$ as above. The first identity in \eqref{eq-chiH} follow directly from \eqref{eq-alm}, and the second identity follows from the subsequent remark.
For \eqref{eq-gH}, by \eqref{eq-g} and \eqref{eq-chiH} we have
 \[ g_H(n)  =\sum_{d_1\mid a}\mu\Big(\frac{a}{d_1}\Big)\chi_{H_{Q_0}}(d_1)
    \cdot\prod_{\ell\nmid Q_0}\sum_{d_\ell\mid \ell^{v_\ell(b)}}\mu\Big(\frac{\ell^{v_\ell(b)}}{d_\ell}\Big)\chi_{H_\ell}(d_\ell)
    =  g_{H_{Q_0}}(a) \prod_{\ell\nmid Q_0}g_{H_\ell}(\ell^{v_\ell(b)}). \]
If for all $\ell\nmid Q_0$ we have $0\in V_\ell$, then $v_\ell(a)=0\in V_\ell$ for every $\ell\nmid Q_0$.  If $1\in H$ and $\ell\nmid Q_0$, then $v_{Q_0}(\ell^v)=0\in v_{Q_0}(H)$ and  $v_q(\ell^v)=0\in V_q$ for all $q\nmid \ell Q_0$. The identities involving $g_H$ then follow from \eqref{eq-gH} and the definition \eqref{eq-g}.
\end{proof}

\begin{lem}\label{lem_g}
Let $H\subset\NN$. We have $g_H(1)=\chi_H(1)$, and for all primes $\ell$ and all $m\geq1$, we have
\begin{equation}\label{eq_gmult2}
g_{H_\ell}(\ell^m)=\chi_{H_\ell}(\ell^m)-\chi_{H_\ell}(\ell^{m-1})=\begin{cases}
-1 & \text{if }  m\notin V_\ell \text{ and } m-1\in V_\ell \\
1 & \text{if } m\in V_\ell \text{ and } m-1\notin V_\ell\\
0 & \text{otherwise}.
\end{cases}
\end{equation}
If $H$ is almost cut by valuations with $Q_0$ as in \eqref{eq-alm} and $1\in H$, then for all primes $\ell\nmid Q_0$ the identities \eqref{eq_gmult2} hold with $g_{H_\ell}$ and $\chi_{H_\ell}$ replaced by $g_H$ and $\chi_H$, respectively.
\end{lem}

\begin{proof}
The first characterization follows from the formula \eqref{eq-g}, whereas the second one is a consequence of the first one, as $\ell^v\in H_\ell$ if and only if $v\in V_\ell$. The last statement follows from Lemma \ref{lem_mult}.
\end{proof}

\subsection{$k$-free and $k(\ell)$-free indices}

Let $S(k(\ell))$ be the set of positive integers which are $k(\ell)$-free, i.e.\ of those integers $n$ such that $v_\ell(n)< k(\ell)$, where $k(\ell)\geq2$, for all $\ell$.

\begin{thm}\label{thm_kl}
We have
\[ \pi_{\mc P_{S(k(\ell)),C}}(x) = \Li(x)\sum_{n\geq1} \frac{\mu(n)c(f_n)}{[F_{f_n,f_n}:K]}+O_{F,K,G,k}\left( x\left( \frac{(\log\log x)^2}{\log x} \right)^{1+\frac{3(k-1)}{2k(r+2)}} \right), \]
where $f_n=\prod_{\ell\mid n}\ell^{k(\ell)}$ and $k:=\min_\ell k(\ell)\geq2$. 
\end{thm}

\begin{proof}
For $S=S(k(\ell))$, in view of \eqref{eq-piS}, it suffices to prove that $g_S$ satisfies  \eqref{eq-conv-new} for some $0<\kappa\leq2$ and apply Theorem \ref{thm_uncond}. 
Since $S$ is cut by valuations with $1\in S$, $g_S$ is multiplicative by Lemma \ref{lem_mult}, and by Lemma \ref{lem_g} we obtain $g_{S}(\ell^m)=-1$  if $m=k(\ell)$,  and $g_{S}(\ell^m)=0$ for all other $m\geq1$. Therefore, $g_{S}(n)=\mu(\rad(n))$ if and only if $n=f_n$, and it is $0$ otherwise, and by definition, we have $n^k\mid f_n$ for all $n$ squarefree. Thus, recalling that  if $d\mid n$, then $\varphi(d)\mid\varphi(n)$ and $\varphi(nd)=\varphi(n)d$, we have
\[  \sum_{n>z} \frac{\av{g_{S}(n)}}{\varphi(n)}
= \sum_{f_n>z} \frac{\mu(n)^2}{\varphi(f_n)}
\leq \sum_{n> z^{1/k}} \frac{\mu(n)^2}{\varphi(n)n^{k-1}}
\ll \frac{1}{z^{1-1/k}},
\]
where we applied Lemma \ref{lem-phi}. We may conclude by Theorem \ref{thm_uncond} with $\kappa=1-1/k$ to obtain the desired asymptotic formula.
\end{proof}

\begin{cor}\label{cor-kfree}
    Let $S(k)$ be the set of $k$-free integers for $k\geq2$. We have
    \[ \pi_{\mc P_{S(k),C}}(x) \sim \Li(x)\sum_{n\geq1} \frac{\mu(n)c(n^k)}{[F_{n^k,n^k}:K]}
   , \]
    with error term as in Theorem \ref{thm_kl}.
\end{cor}

\begin{proof}
    It suffices to take $k(\ell)=k$ for all $\ell$ in Theorem \ref{thm_kl}.
\end{proof}

Both density formulas obtained for $k$-free and $k(\ell)$-free indices were already known under GRH, see \cite[Example 25]{JPS}.

\subsection{The index in sets which are cut by valuations}
In this section we prove two general results for the primes $\p$ of $K$ such that the index of $(G\bmod\p)$ lies in a given set $H$ which is cut by valuations. This is achieved in two cases requiring some relatively strong conditions on $H$.
Let us introduce some notation. Given a set $H$ cut by valuations,  we may write each $V_\ell$ as a disjoint  union of intervals of one of the following two types:
\begin{equation}\label{eq-vl}
    V_\ell = \bigcup_{i\geq 1} [a_i,b_i] \quad \text{ or } \quad  V_\ell = \bigcup_{i= 1}^{N_\ell} [a_i,b_i]  \cup [a_{N_\ell+1},\infty),
\end{equation}
where these intervals are intended as discrete sets of integers,  $a_i\leq b_i<a_{i+1}-1$ for all $i$, and $N_\ell\geq0$. All $a_i,b_i$ depend on $\ell$, and we have $a_1=0$ for all $\ell$ except finitely many. 
Whenever we consider $\min_\ell b_1$, we implicitly take $b_1=\infty$ if $V_\ell =[a_1,\infty)$.
With this notation, for all $\ell$ and all $m\geq0$, by Lemma \ref{lem_g} we have
\begin{equation}\label{eq-gab}
     g_{H_\ell}(\ell^m)= \begin{cases}
   -1 & \text{if } m=b_i+1 \text{ for some } i \\
   1 & \text{if } m=a_i  \text{ for some } i  \\
   0 & \text{otherwise}.
\end{cases}
\end{equation}

\begin{thm}\label{vlfin}
    Let $H$ be cut by valuations, with $1\in H$, and suppose that $V_\ell$ is given as in \eqref{eq-vl} for every $\ell$, and for the first case we suppose that the union is finite and we write $V_\ell = \bigcup_{i= 1}^{N_\ell} [a_i,b_i]$. 
    Assume that for every squarefree integer $n\geq1$ we have $\prod_{\ell\mid n}2N_\ell\ll n^\alpha$  for some $\alpha>0$, and that  $\min_\ell b_1>\alpha$.
    
    We have 
    \[ \pi_{\mc P_{H,C}}(x) = \Li(x)\sum_{n\geq1} \frac{g_H(n)c(n)}{[F_{n,n}:K]}+O_{F,K,G,k,\alpha}\left( x\left( \frac{(\log\log x)^2}{\log x} \right)^{1+\frac{1}{r+2}(1-\frac{\alpha+1}{k})} \right), \]
    where  $k:=1+\min_\ell b_1$, $g_H$ is multiplicative and, for each $\ell$ and $m\geq0$, $g_H(\ell^m)$ is equal to \eqref{eq-gab}.
\end{thm}

\begin{proof}
The function $g_H$ is multiplicative by Lemma \ref{lem_mult}, and it is characterized by Lemma \ref{lem_g} (since $1\in H$, we have $g_H(\ell^m)=g_{H_\ell}(\ell^m)$ for all $\ell$ and $m\geq0$). In particular, $g_H(n)\neq0$ if and only if, for all $\ell$, $v_\ell(n)$ is equal to $a_i$ or $b_i+1$ for some $i$.
Therefore, for all $n$ such that $g_H(n)\neq0$ we have $\rad(n)^k\mid n$, and for a fixed $m=\rad(n)$ we may use this lower bound for at most  $e(m):=\prod_{\ell\mid m}2N_\ell$ integers.
Notice that, if $V_\ell=\Z_{\geq0}$ for some $\ell\mid m$, then $e(m)=0$ and the bound in the statement still holds. In fact, in this case we have $g_H(m)=0$.

By the above and making use of the bound $e(n)\ll n^\alpha$ for $n$ squarefree, we then have
\[  \sum_{n>z} \frac{\av{g_{H}(n)}}{\varphi(n)}
\leq \sum_{n> z^{1/k}} \frac{\mu(n)^2e(n)}{\varphi(n^k)}
\ll \sum_{n> z^{1/k}} \frac{\log\log n}{n^{k-\alpha}} ,
\]
where, since $k>1+\alpha$, the series converges and we have
\[ \sum_{n>z} \frac{\av{g_{H}(n)}}{\varphi(n)} 
\ll_\varepsilon \frac{1}{z^{1-(\alpha+1)/k-\varepsilon}}
\ll_{k,\alpha} z^{-\frac{2}{3}(1-\frac{\alpha+1}{k})}, \]
where we chose $\varepsilon=\frac{1}{3}(1-\frac{\alpha+1}{k})>0$.
Finally, in view of \eqref{eq-piS}, we may apply Theorem \ref{thm_uncond} with $\kappa=\frac{2}{3}(1-\frac{\alpha+1}{k})$.
\end{proof}

\begin{exa}
Let $H$ be cut by valuations and suppose that $V_\ell = \{0,1\}\cup\{s,s+1,\ldots\}$ for all $\ell$, where $s\geq3$. 
We have $\min_\ell b_1=1$ and $\prod_{\ell\mid n}2\leq 2\sqrt{n}$, and hence we may take $\alpha=1/2$ and $k=2>1+\alpha$.
Thus, the conditions of Theorem \ref{vlfin} are satisfied.
The map $g_H$ is multiplicative, and by \eqref{eq-gab} we have
\[ 
 g_H(\ell^m)= \begin{cases}
   -1 & \text{if } m=2 \\
   1 & \text{if } m=0,s \\
   0 & \text{otherwise}.
\end{cases}
\]
Hence $g_H(n)\neq0$ if and only if $n=a^2b^s$ for $a,b$ squarefree with $(a,b)=1$; moreover, for such $n$ we have $g_H(n) = \mu(a)$.
 We can then apply Theorem \ref{vlfin}  to obtain
\[ \pi_{\mc P_{H}}(x) = \Li(x)\sum_{a\geq1}\mu(a)\sum_{b\geq1} \frac{\mu(ab)^2}{[K_{a^2b^s,a^2b^s}:K]}+O\left( x\left( \frac{(\log\log x)^2}{\log x} \right)^{1+\frac{1}{4(r+2)}} \right). \]
\end{exa}

\begin{thm}\label{vlinf}
Let $H$ be cut by valuations, with $1\in H$. Let us suppose that $1,2\in V_\ell$  for all $\ell$. We have
\[ \pi_{\mc P_{H,C}}(x) = \Li(x)\sum_{n\geq1} \frac{g_H(n)c(n)}{[F_{n,n}:K]}+O_{F,K,G}\left( x\left( \frac{(\log\log x)^2}{\log x} \right)^{1+\frac{1}{2(r+2)}} \right), \]
where $g_H$ is multiplicative and, for each $\ell$ and $m\geq0$, $g_H(\ell^m)$ is equal to \eqref{eq-gab}.
\end{thm}

\begin{proof}
Set $k:=1+\min_\ell b_1\geq3$, where $b_1$ is defined by writing $V_\ell$ as in \eqref{eq-vl} for all $\ell$.
We start as in the proof of Theorem \ref{vlfin}. Here we have that $g_H(n)\neq0$ only if $n$ is divisible by $\rad(n)^k$, and we write $n=am^k$ for some $m$ squarefree and $a\mid m^\infty$.

Therefore, also recalling that $\varphi(am^k)=\varphi(m^k)a$ as $a\mid m^\infty$, we have
\[ \sum_{n>z} \frac{\av{g_{H}(n)}}{\varphi(n)} 
\leq \sum_{\substack{m\geq1, a\mid m^\infty \\ am^k> z}} \frac{\mu(m)^2 }{\varphi(am^k)}
\leq \sum_{m\geq1}\frac{1}{\varphi(m^k)} \sum_{\substack{a\mid m^\infty \\ a> z/m^k}}  \frac{1}{a}.  \]
By Lemma \ref{lem-rankin}, the inner series on $a$ is bounded by $m^{1-\rho}(\frac{m^k}{z})^\rho$ with $0<\rho<1$, yielding
\[ 
\sum_{n>z} \frac{\av{g_{H}(n)}}{\varphi(n)} \ll_{\rho,\varepsilon} \frac{1}{z^\rho} \sum_{m\geq1} \frac{m^{1+(k-1)\rho}}{m^{k-\varepsilon}} 
= \frac{1}{z^\rho} \sum_{m\geq1}\frac{1}{m^{(k-1)(1-\rho)-\varepsilon}}.
\]
If the choice of $\rho$ satisfies $\rho<(1-\varepsilon)/2$, then the series in the last step converges. Hence we may take $\rho=1/3$ (with $\varepsilon>0$ sufficiently small) and bound the initial series by $1/z^{1/3}$. Finally, we may apply Theorem \ref{thm_uncond} with $\kappa=1/3$.
\end{proof}

\begin{rem}\label{rem-Z}
Suppose that $H$ is cut by valuations, with $1\in H$, and $V_\ell=\Z_{\geq0}$ for some $\ell$ (possibly for infinitely many $\ell$), i.e.\ there is no restriction on the $\ell$-adic valuation of the elements of $H$ for the considered $\ell$.  In this case, we have $g_H(\ell^m)=0$ for all $m\geq1$ by \eqref{eq-gab}. 
Hence, writing $P$ for the set of primes $\ell$ such that $V_\ell\neq \Z_{\geq0}$, and $P_\infty$ for the set of integers whose prime factors are in $P$, the densities of Theorems \ref{vlfin} and \ref{vlinf} reduce to
\[ \dens(\mc P_{H,C})=\sum_{n\in P_\infty} \frac{g_H(n)c(n)}{[F_{n,n}:K]}. \]
\end{rem}

\subsection{Prescribing $\ell$-adic valuations for finitely many primes}

In this section we consider a special case of sets almost cut by valuations, namely those with $\ell$-adic conditions only for finitely many primes $\ell$.

\begin{thm}\label{thm-HQ}
Let $H\subseteq \mathbb N$ be such that $H=H_Q$ for some $Q\geq1$. Then for all $0<\rho<1$ we have
\[ \pi_{\mc P_{H_Q,C}}(x) = \Li(x)\sum_{ n\mid Q^\infty} \frac{g_{H_Q}(n)c(n)}{[F_{n,n}:K]}+O_{F,K,G,Q,\rho}\left( x\left( \frac{(\log\log x)^2}{\log x} \right)^{1+\frac{3\rho}{2(r+2)}} \right), \]
where $g_{H_Q}$ is as in \eqref{eq-g}.
\end{thm}
The restriction $n\mid Q^\infty$ can be dropped because $g_{H_Q}(n)=0$ if there is a prime $\ell\mid n$ such that $\ell\nmid Q$, as we will see in the proof.

\begin{proof}
The set $H$ is almost cut by valuations (indeed, $H_{\ell}=\NN$ for every $\ell\nmid Q$) and hence $g_H(n)$ decomposes as in \eqref{eq-gH}. By Lemma \ref{lem_g} for all $\ell\nmid Q$   we have $g_{H_\ell}(1)=1$ and $g_{H_\ell}(\ell^ m)=0$ for $m\geq1$, whence $g_H(n)\neq0$ only if $n\mid Q^ \infty$. 
Since by definition, see \eqref{eq-g}, $g_H(n)=\sum_{t\mid n, t\in H} \mu(n/t)$ and there are at most $2^{\omega(n)}$ nonzero terms in this sum, for $n\mid Q^\infty$ we may bound $\av{g_H(n)}\leq 2^{\omega(Q)}$. 
Then we have
\[
\sum_{n>z}\frac{\av{g_H(n)}}{\varphi(n)}  
\leq 2^{\omega(Q)}\cdot \sum_{\substack{n\mid Q^\infty \\ n>z}} \frac{1}{\varphi(n)} \ll \frac{2^{\omega(Q)}Q}{\varphi(Q)} \sum_{\substack{n\mid Q^\infty \\ n>z}} \frac{1}{n}\ll_{Q,\rho} \frac{1}{z^\rho},
\]
where we used the bound $\frac{n}{\varphi(n)}\leq \frac{Q}{\varphi(Q)}$ for $n\mid Q^\infty$, and we applied Lemma \ref{lem-rankin} with $0<\rho<1$. We may conclude by applying Theorem \ref{thm_uncond} with $\kappa=\rho$.
\end{proof}

\begin{exa}\label{exa-l}
Let $\ell$ be a fixed prime number.
Let $H\subseteq\NN$ and consider $H_\ell$, the preimage under $v_\ell$ of $v_\ell(H)$. We may apply Theorem \ref{thm-HQ} to $H_\ell$.
By the proof of Theorem \ref{thm-HQ},  $g_{H_\ell}$ is nonzero only on powers of $\ell$ and it is described by \eqref{eq_gmult2}.
We distinguish two cases as in \eqref{eq-vl}. In the first case of \eqref{eq-vl}  we obtain 
\[ \dens(\mc P_{H_\ell,C}) = \sum_{i\geq1} \bigg( \frac{c(\ell^{a_i})}{[F_{\ell^ {a_i},\ell^ {a_i} }:K]} - \frac{c(\ell^{b_i+1})}{[F_{\ell^ {b_i+1},\ell^ {b_i+1} }:K]} \bigg). \]
This sum is finite if $V_\ell$ is the union of finitely many bounded intervals. In the second case of \eqref{eq-vl} we have
\[ \dens(\mc P_{H_\ell,C}) = \sum_{i=1}^{N_\ell} \bigg( \frac{c(\ell^{a_i})}{[F_{\ell^ {a_i},\ell^ {a_i} }:K]} - \frac{c(\ell^{b_i+1})}{[F_{\ell^ {b_i+1},\ell^ {b_i+1} }:K]} \bigg) + \frac{c(\ell^{a_{N_\ell+1}})}{[F_{\ell^ {a_{N_\ell+1}},\ell^ {a_{N_\ell+1}} }:K]}. \]
Since for a given $v\in V_\ell$ we have
\[ \dens\big(\mc P_{v_\ell^{-1}(v),C}\big) = \frac{c(\ell^v)}{[F_{\ell^ {v},\ell^ {v} }:K]} - \frac{c(\ell^{v+1})}{[F_{\ell^ {v+1},\ell^ {v+1} }:K]} \]
by Chebotarev's density theorem,
we recovered the result of Proposition \ref{prop-l} which provides the identity
\[ \dens(\mc P_{H_\ell,C})=\sum_{v\in V_\ell}\dens\big(\mc P_{v_\ell^{-1}(v),C}\big). \]
\end{exa}

\begin{rem}\label{rem-const}
By Proposition \ref{kummer} there is an integer $B$, depending only on $F,K,G$, such that for all $\ell\nmid B$ we have $[K_{\ell^ {n},\ell^ {n} }:K]=\varphi(\ell^n)\ell^{rn}$ and $K_{\ell,\ell}\cap F=K$. The latter implies that the two conditions on the $K$-automorphisms counted by $c(\ell^n)$, recall its definition, are mutually independent, whence $c(\ell^n)=\sharp C$ for all $n$.
Thus, for $\ell\nmid B$, in the context of Example \ref{exa-l} we derive the identities
\[
\dens\big(\mc P_{v_\ell^{-1}(v),C}\big) = \frac{\sharp C}{[F:K]}E_{v,r}(\ell) \quad \text{ and } \quad
\dens(\mc P_{H_\ell,C}) = \frac{\sharp C}{[F:K]}A_{V_\ell,r}(\ell), 
\]
where, for $r\geq1$ and $v\geq0$, we set 
\[  E_{v,r}(\ell):=\begin{cases}
1-\frac{1}{\ell^r(\ell-1)}& \text{if $v=0$}\\
\frac{1}{\ell^{v(r+1)}}\cdot \frac{\ell}{(\ell-1)} (1-\frac{1}{\ell^{r+1}}) & \text{if $v>0$\,,} 
\end{cases}\]
and 
\[ A_{V_\ell,r}(\ell):= \sum_{v\in V_\ell} E_{v,r}(\ell).  \]
In particular, for trivial Frobenius condition we have
\begin{equation}\label{eq_vHell}
\dens\big(\mc P_{v_\ell^{-1}(v)}\big) = E_{v,r}(\ell) \quad \text{ and } \quad
\dens(\mc P_{H_\ell}) = A_{V_\ell,r}(\ell).
\end{equation}
The first identity in \eqref{eq_vHell} is shown also in \cite[Corollary 21]{JPS}. 
\end{rem}

\begin{exa}
Let $H=n\mathbb N$ for a given integer $n>1$. Then we have $H=H_n$ and $V_\ell=\{v: v\geq v_\ell(n)\}$ for all $\ell$. 
Since $H$ is cut by valuations, by \eqref{eq-gH} we have $g_H(m)=\prod_\ell g_{H_\ell}(\ell^{v_\ell(m)})$, and hence by Lemma \ref{lem_g} we obtain  $g_{H}(n)=1$, and $g_{H}(m)=0$ if $m\neq n$. Thus, by Theorem \ref{thm-HQ} we recover an asymptotic formula for Chebotarev's density theorem applied to the Galois extension $F_{n,n}/K$.
\end{exa}

\begin{cor}
Let $m>1$ be a squarefree integer, $t\mid m^\infty$ and set $S_{m,t}:=\{n\geq1: (n,m^\infty)=t\}$. We have
\[ \pi_{\mc P_{S_{m,t},C}}(x) \sim \Li(x)\sum_{t\mid n\mid mt} \frac{\mu(n/t)c(n)}{[F_{n,n}:K]}. \]
In particular, for the set $S_{m,1}$ of integers which are coprime to $m$, we have
\[ \pi_{\mc P_{S_{m,1},C}}(x) \sim \Li(x)\sum_{n\mid m} \frac{\mu(n)c(n)}{[F_{n,n}:K]}. \]
In both cases, the error term can be taken as in Theorem \ref{thm-HQ} with $Q=m$.
\end{cor}

\begin{proof}
The set $S=S_{m,t}$ is the preimage of $v_m(S)$ under $v_m$. By Theorem \ref{thm-HQ}, we have $g_S(n)\neq0$ only if $n\mid m^\infty$. Since the only element of $S$ dividing $m^\infty$ is $t$, by \eqref{eq-g} we have 
\[ g_S(n) = \begin{cases}
    \mu(n/t) & \text{if } t\mid n\mid mt \\
    0 & \text{otherwise}.
\end{cases} \]
We conclude by applying Theorem \ref{thm-HQ}.
The case $t=1$ follows directly.
\end{proof}

\subsection{The index in sets which are almost cut by valuations}

\begin{thm}\label{prop-almostcut}
Suppose that $H$ is almost cut by valuations,  let $Q\geq1$ be such that $H=H_{Q}\cap H_{Q}'$ where $H'_Q:=\bigcap_{\ell\nmid Q}H_\ell$.
If the function $g_{H_Q'}$ satisfies the estimate
\begin{equation}\label{eq-conv-new2}
    \sum_{\substack{n>z \\ (n,Q)=1}} \frac{|g_{H_Q'}(n)|}{\varphi(n)} \ll \frac{1}{z^\kappa}
\end{equation}
for some $0<\kappa<1$, then we have
\[
    \pi_{\mc P_{H,C}}(x) = \delta_{H,C}\cdot\Li(x)+O\left( x\left( \frac{(\log\log x)^2}{\log x} \right)^{1+\frac{3\kappa}{2(r+2)}} \right),
\]
where 
\begin{equation}\label{eq-almostcut}
    \delta_{H,C}:=\sum_{m\mid Q^\infty}g_{H_Q}(m)\sum_{\substack{n\geq1\\ (n,Q)=1}} \frac{g_{H_Q'}(n)c(mn)}{[F_{mn,mn}:K]}.
\end{equation}
We have $g_{H_Q'}(n)=\prod_{\ell\nmid Q} g_{H_\ell}(\ell^{v_\ell(n)})$ for $n$ such that $(n,Q)=1$, and for $\ell\nmid Q$ and $v\geq1$, $g_{H_\ell}(\ell^v)$ is given by \eqref{eq_gmult2}.
If $0\in V_\ell$ for all $\ell\nmid Q$, then $g_{H_Q}(m)=g_H(m)$ for $m\mid Q^\infty$. 
If $1\in H$, then $g_{H_Q'}(n)=g_H(n)$ for $n$ with $(n,Q)=1$.

The constant implied by the error term depends only on $F,K,G,Q,\kappa$ and possibly any other parameters on which the constant implied by \eqref{eq-conv-new2} might depend.
\end{thm}

\begin{proof}
Notice that the set $H_Q'$ is cut by valuations and by \eqref{eq-gH} with $Q_0=1$ and Lemma \ref{lem_g} we have $g_{H_Q'}(n)=\prod_{\ell\nmid Q} g_{H_\ell}(\ell^{v_\ell(n)})$ for $n$ such that $(n,Q)=1$.
Writing $n=ad$ with $a\mid Q^\infty$ and $(d,Q)=1$, by \eqref{eq-gH} we have
$ g_H(n)= g_{H_{Q}}(a) g_{H_{Q}'}(d)$.
Thus
\begin{equation}\label{eq-thmalm}
    \sum_{n>z} \frac{\av{g_H(n)}}{\varphi(n)} 
= \sum_{a\mid Q^\infty} \frac{|g_{H_Q}(a)|}{\varphi(a)} \sum_{\substack{d>z/a \\ (d,Q)=1}} \frac{|g_{H_Q'}(d)|}{\varphi(d)}.
\end{equation}
By the assumption, the inner series can be bounded by $(a/z)^\kappa$. Notice that for $a\mid Q^\infty$ we have $|g_{H_Q}(a)|\leq 2^{\omega(Q)}$ by \eqref{eq-g}.
Therefore,
\[ \sum_{n>z} \frac{\av{g_H(n)}}{\varphi(n)} 
\ll_{Q} \frac{1}{z^{\kappa}} \sum_{a\mid Q^\infty} \frac{a^{\kappa}}{\varphi(a)} 
\ll \frac{1}{z^{\kappa}}\frac{Q}{\varphi(Q)} \sum_{a\mid Q^\infty} \frac{a^{\kappa}}{a},
\]
where we used the bound $\frac{a}{\varphi(a)}\leq \frac{Q}{\varphi(Q)}$ for $a\mid Q^\infty$. The inner series can be expressed as the (finite) product
\[ \prod_{\ell\mid Q}\frac{\ell^{1-\kappa}}{\ell^{1-\kappa}-1}. \]
Hence, the series \eqref{eq-thmalm} is $O_{Q}(\frac{1}{z^{\kappa}})$, and we may conclude by Theorem \ref{thm_uncond}. 
\end{proof}

\begin{cor}\label{cor-almostcut}
Writing $V_\ell$ as in \eqref{eq-vl} and assuming $a_1=0$ for each $\ell\nmid Q$, in Theorem \ref{prop-almostcut} we may take:
\begin{enumerate}[(i)]
    \item $\kappa=1/3$, if $\min_{\ell\nmid Q}b_1\geq2$;
    \item $\kappa=\frac{2}{3}(1-\frac{1+\alpha}{k})$, if we have:  $V_\ell=\bigcup_{i=1}^{N_\ell}[a_i,b_i]$ in the first case of \eqref{eq-vl} for $\ell\nmid Q$; for all $n$ with $(n,Q)=1$ we have $\prod_{\ell\mid n}2N_\ell\ll n^\alpha$ for some $\alpha>0$; $k:=1+\min_{\ell\nmid Q}b_1>1+\alpha$.
\end{enumerate}
\end{cor}

\begin{proof}
    With the setup of Theorem \ref{prop-almostcut}, suppose that we are in one of the two cases (i) or (ii). Since $H_Q'$ is cut by valuations with $1\in H_Q'$ (we have $a_1=0$ for all $\ell$),  letting $\kappa$ be as in the respective case, we may bound the inner series of \eqref{eq-thmalm} by $(a/z)^\kappa$ as in the proofs of Theorems \ref{vlinf} and \ref{vlfin}, respectively.
\end{proof}

\begin{rem}\label{rem-Z2}
    In the context of Theorem \ref{prop-almostcut}, if for some $\ell\nmid Q$ we have $V_\ell=\Z_{\geq0}$, then reasoning as in Remark \ref{rem-Z} we may restrict the indices $n$ in \eqref{eq-almostcut} to $n\in P_\infty$, where $P$ is the set of primes not dividing $Q$ and such that $V_\ell\neq\Z_{\geq0}$. In particular, requiring $V_\ell=\Z_{\geq0}$ for all $\ell\nmid Q$ reduces the problem to Theorem \ref{thm-HQ}, which can then be obtained as a corollary of Theorem \ref{prop-almostcut} (except for the exponent in the error term). Indeed, in this case the convergence condition \eqref{eq-conv-new2} is trivial as $g_{H_Q'}(n)=0$ for all $n>1$.
\end{rem}

\begin{rem}\label{rem-relax}
Let $H$ be cut by valuations, and let us relax the assumptions of Theorems \ref{vlfin} and \ref{vlinf}  only for finitely many primes (whose product we denote by $Q$) as follows:
    \begin{enumerate}[(i)]
        \item possibly $1\notin H$, but $0\in V_\ell$ for all primes $\ell\nmid Q$;
        \item in Theorem \ref{vlfin},  we assume that $V_\ell$ is written as in \eqref{eq-vl} for all $\ell\nmid Q$ with $V_\ell=\bigcup_{i=1}^{N_\ell}[a_i,b_i]$ in the first case, $\prod_{\ell\mid n}2N_\ell\ll n^\alpha$ holds for all $n$ coprime to $Q$, and $k:=1+\min_{\ell\nmid Q}b_1>1+\alpha$;
        \item  in Theorem \ref{vlinf}, we assume that $1,2\in V_\ell$ for all $\ell\nmid Q$.
    \end{enumerate}
Then we may consider the set $H=H_{Q}\cap\bigcap_{\ell\nmid Q}H_\ell$ as almost cut by valuations and apply Theorem \ref{prop-almostcut} for (i) and Corollary \ref{cor-almostcut} for (ii) and (iii) to obtain an asymptotic formula for $\pi_{\mc P_{H,C}}(x)$.
\end{rem}

\section{The density as a limit and as an Euler product}
\label{sec-prod}

The following result is already known in a more general form under the assumption of GRH, see \cite[Theorem 11]{JPS}.
\begin{prop}\label{prop-lim}
Suppose that $H$ is almost cut by valuations, and such that $g_H$ satisfies the convergence condition of Theorem \ref{thm_uncond}.
We have
\[ \dens(\mc P_{H,C}) = \lim_{Q\to\infty}\dens(\mc P_{H_Q,C}). \]
\end{prop}

\begin{proof}
By Theorem \ref{thm-HQ} we have
\[ \dens(\mc P_{H_Q,C})=\sum_{n\geq1}\frac{g_{H_Q}(n)c(n)}{[F_{n,n}:K]}. \]
Let $Q_0$ be such that $H=H_{Q_0}\cap\bigcap_{\ell\nmid Q_0}H_\ell$. By Lemma \ref{lem_mult} if $Q$ is divisible by $nQ_0$ and by the finitely many primes $\ell$ such that $0\notin V_\ell$, then we have $g_{H_Q}(n)=g_{H}(n)$. Thus, for every $n$ we have $g_{H_Q}(n)\to g_{H}(n)$ when $Q\to\infty$, by which we mean the limit $x\to\infty$ for $Q_x:=\prod_{\ell\leq x}\ell$. 
 
Recall that $c(n)$ is independent of $H_Q$.
Then the double sequence
$\sum_{n=1}^N\frac{c(n)}{[F_{n,n}:K]}g_{H_Q}(n)$
has a limit when $Q\to\infty$ and when $N\to\infty$, separately. By exchanging the order of the limits, we have
\[ \lim_{Q\to\infty}\sum_{n\geq1}\frac{g_{H_Q}(n)c(n)}{[F_{n,n}:K]}
= \lim_{N\to\infty} \lim_{Q\to\infty} \sum_{n=1}^N\frac{g_{H_Q}(n)c(n)}{[F_{n,n}:K]} =
\sum_{n\geq1}\frac{g_{H}(n)c(n)}{[F_{n,n}:K]}.  \]
The last series equals $\dens(\mc P_{H,C})$ by Remark \ref{rem-dens}.
\end{proof}

We introduce the following notation for the next theorem.
For $Q\geq1$  we define
\begin{equation}\label{eq_def_CHQ}
C_{H,Q}:= \bigcup_{h\in H} C_{h_Q} = \bigsqcup_{h_Q} C_{h_Q} \subseteq \Gal(FK_{Q^\infty,Q^\infty}/K) ,
\end{equation}
where $h_Q:=\gcd(h,Q^\infty)$, $K_{Q^\infty,Q^\infty}$ is the union of all number fields $K_{Q^e,Q^e}$ for $e\geq1$, and $C_{h_Q}$ is the conjugacy-stable set of those $K$-automorphisms $\sigma$ which satisfy the following conditions:
\begin{itemize}
\item $\sigma|_F\in C$,
\item  $\sigma$ fixes $K_{h_Q,h_Q}$,
\item for all prime numbers $q\mid Q$, $\sigma$ does not fix $K_{qh_Q,qh_Q}$.
\end{itemize} 
Let $\mu_{\Haar}$ be the Haar measure on $\Gal(FK_{Q^\infty,Q^\infty}/K)$. By \eqref{eq_def_CHQ} we have 
\begin{equation}\label{eq-HaarQ}
    \dens(\mc P_{H_Q,C}) = \mu_{\Haar}(C_{H,Q}) = \sum_{h_Q}\mu_{\Haar}(C_{h_Q}).
\end{equation}

\begin{thm}\label{thm_prod}
Suppose that $H$ is almost cut by valuations and such that $g_H$ satisfies the convergence condition \eqref{eq-conv-new}. Then there is an integer $B\geq1$, depending only on $F$, $K$, $G$, and with $Q_0\mid B$, where $Q_0$ is associated to $H$ as in \eqref{eq-alm}, such that 
\begin{equation}\label{eq-eulprod}
    \dens(\mc P_{H,C}) = A_{H,r} \cdot \mu_{\Haar}(C_{H,B}) \cdot \prod_{\ell\mid B} A_{V_\ell,r}(\ell)^{-1},
\end{equation}
where $A_{V_\ell,r}(\ell)$ was defined in Remark \ref{rem-const}, and $A_{H,r}:=\prod_\ell A_{V_\ell,r}(\ell)$.
\end{thm}

Under the assumption of GRH, the formula \eqref{eq-eulprod} holds for any set $H$ which is almost cut by valuations, see \cite[Theorem 24]{JPS}. Unconditionally, since if $g_H$ satisfies the condition of Theorem \ref{thm_uncond}, then $\dens(\mc P_{H,C})$ exists, it must be given by the formula \eqref{eq-eulprod}. Here we reprove this statement under the assumptions of Theorem \ref{thm_uncond}, also for the convenience of the reader. The main difference with the original proof is applying Proposition \ref{prop-lim} instead of \cite[Theorem 11]{JPS}.

\begin{proof}[Proof of Theorem \ref{thm_prod}]
By Proposition \ref{kummer}(b) there is an integer $B$, which depends only on $F,K,G$, such that  by varying $\ell\nmid B$, the conditions on the $K$-automorphisms $\sigma$ being the identity on $K_{\ell^v,\ell^v}$ and not on $K_{\ell^{v+1},\ell^{v+1}}$ are pairwise independent, and also independent from the condition $\sigma|_F\in C$. 
Up to replacing $B$ with a multiple, we may suppose that $Q_0\mid B$, so that for $h\in H$ the possible values $h_\ell$ for $\ell\nmid B$ are independent of the integers $h_B$. Hence, for all  $Q>1$ with $B\mid Q$, we have  
\[ C_{H,Q}=C_{H,B}\oplus\bigoplus_{\ell\nmid B,\ell\mid Q}C_{H,\ell}', \]
where $C_{H,\ell}':=\bigsqcup_{h_\ell}C_{h_\ell}'$ and $C_{h_\ell}'$ is defined as $C_{h_\ell}$ by dropping the condition $\sigma|_F\in C$.
By Proposition \ref{prop-lim} and \eqref{eq-HaarQ}, we obtain
\[
\dens(\mc P_{H,C}) = \mu_{\Haar}(C_{H,B}) \cdot \prod_{\ell\nmid B}\mu_{\Haar}(C'_{H,\ell}).
\]
For $\ell\nmid B$ and $v\in V_\ell$, by \eqref{eq_vHell} we have
\[ \mu_{\Haar}(C_{\ell^v}') = \dens(\mc P_{v_\ell^{-1}(v)}) = E_{v,r}(\ell), \]
whence, invoking \eqref{eq-HaarQ}, we have
\[ \mu_{\Haar}(C_{H,\ell}') = \sum_{v\in V_\ell} \mu_{\Haar}(C_{\ell^v}') =  A_{V_\ell, r}(\ell). \]
The constants $A_{V_\ell, r}(\ell)$ are strictly positive by \cite[Proposition 22]{JPS}, and we may conclude.
\end{proof}

An analogous result can be obtained in an alternative way.

\begin{prop}\label{prop_prod}
Suppose that $H$ is almost cut by valuations, and such that the convergence condition \eqref{eq-conv-new2} is satisfied. There is an integer $B\geq1$, depending only on $F,K,G$, and with $Q\mid B$, where $Q$ is as in Theorem \ref{prop-almostcut}, such that 
\begin{equation}\label{eq-prod2}
    \dens(\mc P_{H,C})=\sum_{n\mid B^\infty} \frac{g_{H_B}(n)c(n)}{[F_{n,n}:K]} 
\prod_{\ell\nmid B} \bigg( \chi_{H_\ell}(1)\Big( 1-\frac{1}{(\ell-1)\ell^r} \Big) + \frac{\ell^{r+1}-1}{(\ell-1)\ell^r}\sum_{\substack{v\geq1 \\ v\in V_\ell}}\frac{1}{\ell^{v(r+1)}}   \bigg).
\end{equation}
 In particular, we have
\[ 
\dens(\mc P_{H,C}) = A_{H,r} \cdot \sum_{n\mid B^\infty} \frac{g_{H_B}(n)c(n)}{[F_{n,n}:K]}  \cdot \prod_{\ell\mid B} A_{V_\ell,r}(\ell)^{-1}. 
\]
\end{prop}

\begin{proof}
By the assumption, $\dens(\mc P_{H,C})$ is given by the formula \eqref{eq-almostcut}, where we may replace $Q$ with a multiple, namely $B$. 
Writing $n=ab$ with $a\mid B^\infty$ and $(b,B)=1$, we have $[F_{n,n}:K]=[F_{a,a}:K]\varphi(b)b^r$ by Proposition \ref{kummer}(c). 
Moreover, by Proposition \ref{kummer}(b) we have $K_{b,b}\cap F_{a,a}=K$, implying $c(n)=c(a)$, because every $K$-automorphism counted by $c(n)$ must be the identity on $K_{b,b}$, and this requirement  is independent from any other condition on their restriction to $F_{a,a}$.
Therefore,
\[ 
\dens(\mc P_{H,C}) = \sum_{a\mid B^\infty} \frac{g_{H_B}(a)c(a)}{[F_{a,a}:K]}
\sum_{\substack{b\geq1\\ (b,B)=1}} \frac{g_{H'_B}(b)}{\varphi(b)b^r}.
\]
The inner sum over $b$ can be transformed into the Euler product
\begin{equation}\label{eq-eulprod2}
    \prod_{\ell\nmid B} \sum_{v\geq 0} \frac{g_{H_\ell}(\ell^v)}{\varphi(\ell^v)\ell^{vr}}.
\end{equation}
By Lemma \ref{lem_g} and recalling that $\chi_{H_\ell}(\ell^v)=1$ if and only if $v\in V_\ell$, for $\ell\nmid B$ we have
\begin{align*}
    \sum_{v\geq 0} \frac{g_{H_\ell}(\ell^v)}{\varphi(\ell^v)\ell^{vr}}  & = \chi_{H_\ell}(1)+\frac{\ell}{\ell-1}\sum_{v\geq 1} \frac{\chi_{H_\ell}(\ell^v)-\chi_{H_\ell}(\ell^{v-1})}{\ell^{v(r+1)}} \\
    &= \chi_{H_\ell}(1)\Big( 1-\frac{1}{(\ell-1)\ell^r} \Big) + \frac{\ell}{\ell-1}\Big( 1-\frac{1}{\ell^{r+1}} \Big) \sum_{\substack{v\geq1 \\ v\in V_\ell}}\frac{1}{\ell^{v(r+1)}}.
\end{align*}
Plugging this expression into  the product \eqref{eq-eulprod2}, the initial formula yields the result. For the last part of the statement it suffices to notice that the obtained expression for \eqref{eq-eulprod2} is equal to $\prod_{\ell\nmid B}A_{V_\ell,r}(\ell)$.
\end{proof}

\begin{rem}
    In the context of the previous proposition, writing $V_\ell$, for $\ell\nmid B$, as in \eqref{eq-vl} (with $a_1=0$ for all $\ell\nmid B$), we may use the formula \eqref{eq-gab} to express the product \eqref{eq-eulprod2}. For instance, if $a_1=0$ and $b_1\geq1$ for all $\ell\nmid B$, we may derive the alternative formula
    \[
 \dens(\mc P_{H,C})=\sum_{n\mid B^\infty} \frac{g_{H_B}(n)c(n)}{[F_{n,n}:K]} 
\prod_{\ell\nmid B} \bigg( 1+\frac{\ell}{\ell-1}\Big(\sum_{i\geq2}\frac{1}{\ell^{a_i(r+1)}} - \sum_{i\geq1} \frac{1}{\ell^{(b_i+1)(r+1)}} \Big)  \bigg),
\]
where the range of $i$ in each sum depends on the decomposition of each $V_\ell$.
\end{rem}

\begin{rem}
By Theorem \ref{thm-HQ} the sum in the formula \eqref{eq-prod2} corresponds to $\dens(\mc P_{H_B,C})$, see also \eqref{eq-HaarQ}. If $H$ is cut by valuations, then we may take $Q=1$ in Proposition \ref{prop_prod}, and for $D\geq1$ squarefree, $\dens(\mc P_{H_D,C})$ is given by \eqref{eq-prod2} with the inner product running through the primes $\ell\mid D/(D,B)$.
\end{rem}

\begin{exa}
Consider the set $S(k(\ell))$ of $k(\ell)$-free integers and keep the notation of Theorem \ref{thm_kl}. By Theorem \ref{thm_prod} or Proposition \ref{prop_prod} there is an integer $B\geq1$ (depending only on $F,K,G$) such that
\[ \dens(\mc P_{S(k(\ell)),C}) = A_{H,r}\cdot \frac{\sharp C_{B}}{[F_{f_B,f_B}:K]} \cdot \prod_{\ell\mid B} A_{V_\ell,r}(\ell)^{-1} \]
where $A_{H,r}=\prod_\ell A_{V_\ell,r}(\ell)$, with
\[ A_{V_\ell,r}(\ell) = 1-\frac{1}{(\ell-1)\ell^{k(\ell)(r+1)-1}},  \]
and  $C_{B}\subseteq\Gal(F_{f_B,f_B},K)$ consists of all $K$-automorphisms which are not the identity on $K_{\ell^{k(\ell)},\ell^{k(\ell)}}$ for each $\ell\mid B$, and whose restriction to $F$ lies in $C$. Moreover, we have
\[ \frac{\sharp C_{B}}{[F_{f_B,f_B}:K]} = \sum_{n\mid B}\frac{\mu(n)c(f_n)}{[F_{f_n,f_n}:K]}. \]
Notice that this example is illustrated under the assumption of GRH in \cite[Example 25]{JPS}.
\end{exa}

\section*{Acknowledgments}
The author is supported by the grant no.\ SRG2425-05 of the School of Mathematics and Physics at XJTLU. The author is grateful to Francesco Pappalardi for insightful discussions and to Pieter Moree for valuable feedback.

\end{document}